\documentclass[11pt,reqno]{amsart}
\usepackage{amsmath,amsthm,amsfonts,amssymb,mathrsfs,bm,graphicx,stmaryrd}

\usepackage{mathtools}
\usepackage{dsfont}
\usepackage{multicol}
\usepackage[colorlinks=true,linkcolor=blue]{hyperref}
\hypersetup{bookmarksdepth=1}
\usepackage{enumitem}
\usepackage{bbm}
\usepackage{mathrsfs}
\usepackage{versions}
\usepackage{subcaption}
\usepackage{tikz}
\usepackage[letterpaper,hmargin=1.0in,vmargin=1.0in]{geometry}
\parskip 	\smallskipamount

\newcommand{\E}{\mathbb{E}}

\newcommand{\N}{\mathbb{N}}

\newcommand{\C}{\mathbb{C}}

\def\l{\left}
\def\r{\right}
\def\<{\langle}
\def\>{\rangle}

\newtheorem{theorem}{Theorem}[section]

\newtheorem{proposition}[theorem]{Proposition}

\theoremstyle{remark}

\numberwithin{equation}{section}

\usepackage{palatino}

\title[Monotonicity of average singular values of gaussian matrices]{On the monotonicity of average singular values of complex Gaussian random matrices}
\author[Baslingker]{Jnaneshwar Baslingker}
\author[Dan]{Biltu Dan}

\begin{document}
\begin{abstract}
 We prove that the normalized $k$-th moment of the Laguerre unitary ensemble, as a function of the dimension, is decreasing for fixed $k\in(0,1)\cup (2,\infty)$ and increasing for fixed $k\in(1,2).$ As a corollary, this implies that the average singular value of complex Gaussian random matrix decreases with dimension. Conjectured by Bandeira, Kennedy, and Singer \cite{BKS16}, this property was recently established by Hutn\'{\i}k \cite{H26}, for non-negative integer shape parameter $\alpha$. Our result extends this monotonicity to real $\alpha\geq 0$. 
\end{abstract}
\address{Jnaneshwar Baslingker,  Department of Mathematics, University of Toronto, Canada.}
\email{j.baslingker@utoronto.ca}
\address{Biltu Dan,  Department of Mathematics \& Statistics, Indian Institute of Science Education and Research Kolkata, India.}
\email{bdan@iiserkol.ac.in}
\maketitle

\section{Introduction and main result}
Let $X$ be an $n\times m$ matrix with i.i.d. standard complex normal random variables $\mathcal{C}\mathcal{N}(0,1)$, where $m\geq n$. The joint distribution of the eigenvalues $\lambda_1\geq \lambda_2\geq \dots\geq \lambda_n$ of the complex Wishart matrix $W_{n,m}=XX^*$ is given by \cite{forrester2010log},
\begin{align}\label{eq: laguerre density}
    \frac{1}{Z_{n,m}} \exp{\left(-\sum\limits_{i=1}^n \lambda_i\right)}\prod_i\lambda_i^{(m-n)}\prod_{i<j}(\lambda_i-\lambda_j)^2.
\end{align}
Denote the singular values of $X$ (positive square roots of eigenvalues of $XX^*$) as 
\begin{align*}
    \sigma_1(X)\geq \sigma_2(X)\geq  \dots\geq\sigma_n(X). 
\end{align*}
Let $\alpha=m-n\geq 0$ be fixed. We are interested in the monotonicity of the quantity
\begin{align*}
    C_\alpha(n)=\frac{1}{n^{3/2}}\E\l[\sum\limits_{i=1}^n\sigma_i(X)\r], \quad n\in\N.
\end{align*}
It was conjectured by Bandeira, Kennedy and Singer \cite[Conjecture 8]{BKS16}, that $C_0(n+1)<C_0(n)$. It arises from their analysis of approximation ratios for the little Grothendieck
problem over the unitary group. The conjecture was originally proved by Hutn\'{\i}k \cite{H26} using a recurrence relation
of Abreu \cite{abreu2026recurrence}, obtained from the Christoffel–Darboux formula together with a Tur\'{a}n determinant for Laguerre polynomials. In fact, Hutn\'{\i}k proved that $C_\alpha(n+1)<C_\alpha(n)$ for any fixed $\alpha\in \N\cup \{0\}$. We provide a short and alternate proof of the above-mentioned monotonicity, as a special case of a more general monotonicity result. We use a recurrence relation of the moments of the matrix $W_{n,m}$ from \cite{CMOS19}.

The moments of the matrix $W_{n,m}$, also referred to as Laguerre unitary ensemble, have been well studied \cite{CMOS19, HT03}. Let 
\begin{align*}
    Q_{k,\alpha}(n)=\E\l[\text{Tr}(W_{n,m}^k)\r]=\E\l[\sum\limits_{i=1}^n\lambda_i^k\r],
\end{align*}
for all $k\in\C$, for which the moments exist. We define
\begin{align}\label{eq: C_k def}
    C_{k,\alpha}(n)=\frac{1}{n^{1+k}}Q_{k,\alpha}(n).
\end{align}
In particular, we have $C_{\alpha}(n)=C_{\frac{1}{2},\alpha}(n)$.

Although there is no underlying $X X^*$ matrix model, with independent Gaussians, for non-integer $\alpha$, the joint eigenvalue density \eqref{eq: laguerre density} remains well-defined for any real $\alpha\geq 0$ (see \cite{forrester2010log}). So $C_{k,\alpha}(n)$ in \eqref{eq: C_k def} can still be naturally defined for any $\alpha \geq 0$. 

Our primary investigation is motivated by the specific case $k = \frac{1}{2}$. This case was proved, for any $\alpha\geq 0$, in the first arXiv
version of the present work \cite{BD26}. Our work also extended the result of  Hutn\'{\i}k \cite{H26} for non-negative integer parameter $\alpha$ to all real $\alpha\geq 0$. In this version, we extend the monotonicity result to all real $k>0$. While finalizing this version of the paper, Hutn\'{\i}k \cite{H26N} substantially revised and expanded the monotonicity result to real moments $k$ and shape parameter $\alpha\geq 0$. The revised version \cite{H26N} also establishes the monotonicity properties in Theorem \ref{thm: monotonicity of avg. singular values}. The proof of the result in \cite{H26N} is also based on the recursion \eqref{eq: moment recursion}. Here we give a short proof based on the moment recurrence and elementary convexity arguments. Abreu and Patil \cite{AP26} have recently sharpened the
square case $(k, \alpha) = (1/2, 0)$ with explicit finite-dimensional bounds and
complete asymptotic expansions. Note that the following result, for $k=\frac{1}{2}$, recovers the monotonicity result conjectured by Bandeira, Kennedy and Singer \cite[Conjecture 8]{BKS16}.

\begin{theorem}\label{thm: monotonicity of avg. singular values} For fixed $k> 0$ and $\alpha\geq 0$, we have the following. \hfill
\begin{enumerate}[style=nextline]
    \item For $k\in(0,1)\cup (2,\infty)$ and $\alpha\geq 0$, the sequence $C_{k, \alpha}(n)$ is strictly decreasing. 
    \item For $k\in(1,2)$, the sequence $C_{k, \alpha}(n)$ is strictly increasing for $\alpha=0$ and strictly decreasing for $\alpha>\frac{k^2 + k + 2 - 2^{k+1}}{2^{k+1} - 2}$.
     \item For $k=1,2$, the sequence $C_{k, \alpha}(n)$ is constant for $\alpha=0$ and strictly decreasing for $\alpha>0$.
\end{enumerate} 
\end{theorem} 

\begin{proof}
We use the following recursion from \cite{CMOS19} (See the remark following Theorem $4.4$ of \cite{CMOS19}. The recursion can be proved using Theorem $4.4$ of \cite{CMOS19} and Equation $2.2$ of \cite{GMM91}). For fixed $\alpha\ge 0$ and $k\in \C$ with $\mathrm{Re}(k)>-\alpha-1$, write $Q_{k,\alpha}(n)=a_nR_n$, where $a_n=n(n+\alpha)$. Then
\begin{align}\label{eq: moment recursion}
    (k-1)(k+2)R_n=a_{n+1}R_{n+1}-(a_{n+1}+a_{n-1})R_n+a_{n-1}R_{n-1}.
\end{align}
 Fix $\alpha \ge 0$ and $k>0$. Using \eqref{eq: moment recursion} we get
 \begin{align*}
    Q_{k,\alpha}(n+1) &= \frac{\left(a_{n+1}+a_{n-1}+(k-1)(k+2)\right)}{a_n}{ Q_{k,\alpha}(n)}- Q_{k,\alpha}(n-1).
\end{align*}
This implies
\begin{align}\label{eq: normalized recursion}
    C_{k,\alpha}(n+1)=A_{n,k,\alpha}C_{k,\alpha}(n)-B_{n,k}C_{k,\alpha}(n-1),
\end{align}
 where 
\begin{align*}
A_{n,k,\alpha}:=\l(2+\frac{k(k+1)}{n(n+\alpha)}\r)\l(\frac{n}{n+1}\r)^{1+k} \quad B_n:=\l(\frac{n-1}{n+1}\r)^{1+k}.
\end{align*}
Now define $\Delta_{k,\alpha}(n):=C_{k,\alpha}(n+1)-C_{k,\alpha}(n)$. Then using \eqref{eq: normalized recursion} we have
\begin{align}\label{eq: recursion final}
    \Delta_{k,\alpha}(n)&=(A_{n,k,\alpha}-1)C_{k,\alpha}(n)-B_{n,k}C_{k,\alpha}(n-1)\nonumber\\
    &=(A_{n,k,\alpha}-B_{n,k}-1)C_{k,\alpha}(n)+B_{n,k}\Delta_{k,\alpha}(n-1).
\end{align}
 
 We have
\begin{align}\label{eq: AB def}
 A_{n,k,\alpha}-B_{n,k}-1 =\l(\frac{n}{n+1}\r)^{1+k} \l(\l( 2+ \frac{k(k+1)}{n(n+\alpha)}\r)-\l( 1-\frac1n\r)^{1+k}-\l( 1+\frac1n\r)^{1+k}\r). 
\end{align}
We first consider $\alpha=0$. For $x\in[0,1)$, consider the function \[f_k(x)=(1+x)^{1+k}+(1-x)^{1+k}-2-k(k+1)x^2.\] We have
\[f_k''(x)=k(k+1)\l((1+x)^{k-1}+(1-x)^{k-1}-2\r).\] Note that the function $t\rightarrow t^r$ is strictly convex on $(0,\infty)$, for $-1<r<0$ and $r>1$. It is also strictly concave for $0<r<1$.  
 For $0<k< 1$ and $k> 2$, by strict Jensen's inequality we get, $(1+x)^{k-1}+(1-x)^{k-1}>2$ and hence $f_k''(x)> 0$. We also have $f_k^{(p)}(0)=0$ for $p=0,1,2$. These imply $f_k'(x),f_k(x)>0$. On the other hand, for $1<k<2$, we have $f_{k}''(x)<0$. Hence $f_k'(x),f_k(x)<0$. For $k=1,2$, we have $f_k(x)=0$. As a result we conclude that,
\begin{align}\label{eq: first term sign check}
A_{n,k,0}-B_{n,k}-1\quad 
    \begin{cases}
< 0, & 0<k < 1\text{ and } k>2, \\
= 0,   & k \in\{ 1,2\},\\
>0, &1<k<2.
\end{cases}
\end{align}

Next, from Equation $(4.11)$ of \cite{CMOS19} we have

\begin{align}\label{eq: gen def}
    Q_{k,\alpha}(n)
=
\frac{\left(k+\alpha\right)!\,n(n+\alpha)}
{(1+\alpha)!}
\,{}_3F_2
\left(
\begin{matrix}
1-k,2+k,1-n\\
2,2+\alpha
\end{matrix}
;1
\right),
\end{align}

where ${}_3F_2$ is the generalized hypergeometric function and $\left(r+\alpha\right)!:=\Gamma(r+\alpha+1)$ for any $r\ge 0$.
Hence
\begin{align}\label{eq: 1st value def}
Q_{k,\alpha}(1)
&=
\frac{\left(k+\alpha\right)!\,(1+\alpha)}
{(1+\alpha)!}
\,{}_3F_2
\left(
\begin{matrix}
1-k,2+k,0\\
2,2+\alpha
\end{matrix}
;1
\right)=
\frac{\left(k+\alpha\right)!}{\alpha!},
\end{align}
and
\begin{align*}
Q_{k,\alpha}(2)&=
\frac{\left(k+\alpha\right)!}{(1+\alpha)!}
\,2(2+\alpha)
\,{}_3F_2
\left(
\begin{matrix}
1-k,2+k,-1\\
2,2+\alpha
\end{matrix}
;1
\right)=\frac{\left(k+\alpha\right)!}{(1+\alpha)!}
\,2(2+\alpha)
\left(
1-\frac{(1-k)(2+k)}{2(2+\alpha)}
\right).
\end{align*}
Therefore
  \begin{align*}
  \Delta_{k,\alpha}(1)=\frac{(k+\alpha)!}{(1+\alpha)!}\l(\frac{(k^2+k+2-2^{1+k}+2\alpha-2^{1+k}\alpha}{2^{1+k}}\r).
  \end{align*}
  As a consequence we have
  \begin{align}\label{eq: second term sign changes}
\Delta_{k,\alpha}(1)\quad 
    \begin{cases}
< 0, & 0<k < 1\text{ and } k>2 \text{ and } \alpha\geq 0, \\
= 0,   & k \in\{ 1,2\}\text{ and }\alpha=0,\\
<0, & k \in\{ 1,2\}\text{ and }\alpha>0,\\
>0, &1<k<2 \text{ and } \alpha= 0.
\end{cases}
\end{align}
   Note that $C_{k,\alpha}(n)> 0$ for all $n\in\N$ and $A_{n,k,\alpha}-B_{n,k}-1$ strictly decreases with $\alpha$. Therefore by \eqref{eq: recursion final}, \eqref{eq: first term sign check} and \eqref{eq: second term sign changes}, we obtain $(1),(3)$ and the first part of $(2)$ in Theorem \ref{thm: monotonicity of avg. singular values}. 
  
  We now prove the second part of $(2)$. Fix $1<k<2$.
  Let $\alpha_0=\frac{k^2 + k + 2 - 2^{k+1}}{2^{k+1} - 2}>0$. For $\alpha> \alpha_0$, we have $\Delta_{k,\alpha}(1)< 0$. Now, by \eqref{eq: recursion final}, it is enough to show $A_{n,k,\alpha}-B_{n,k}-1< 0$ for $\alpha> \alpha_0$ and $n\geq 2$.
  Define 
  \[
  g_k(x)=\frac{(1+x)^{k+1}+(1-x)^{k+1}-2}{x^2}, \quad 0< x\leq 1
  \]
   and $g_k(0)=k(k+1)$. One can show that $g_k''(x)<0$ on [0,1] and hence $g_k$ is strictly concave. Therefore
  \begin{align*}
      g_k(x)&>(1-x)g_k(0)+xg_k(1)
      =(2^{k+1}-2)(1+\alpha_0-\alpha_0x)
  \end{align*}
  As $1+\alpha_0-\alpha_0x>\frac{1+\alpha_0}{1+\alpha_0x}=\frac{k(k+1)}{(2^{k+1}-2)(1+\alpha_0x)}$, we have
  \[
  (1+x)^{k+1}+(1-x)^{k+1}-2> \frac{k(k+1)x^2}{1+\alpha_0x} \quad \text{for }0< x< 1.
  \]
   As $A_{n,k,\alpha}-B_{n,k}-1$ strictly decreases with $\alpha$, the above implies $A_{n,k,\alpha}-B_{n,k}-1< 0$ for $\alpha >\alpha_0$ and $n\geq 2$.
\end{proof}
 Hutn\'{\i}k also proved in a different work \cite{H26R} that average singular values of real square Gaussian matrices increases with dimension, in contrast with the complex case. This was also conjectured in \cite{BKS16}. We obtain a short proof of the following bound which was crucial in the earlier version of \cite{H26R}. The bound could be of independent interest.

\begin{proposition}
    $C_{\frac{1}{2},0}(n)-C_{\frac{1}{2},0}(n+1)<\frac{1}{20n^2}$ for $n\geq 3$.
\end{proposition}
\begin{proof}
We retain the notations from the proof of Theorem \ref{thm: monotonicity of avg. singular values}. It follows from Theorem \ref{thm: monotonicity of avg. singular values} and \eqref{eq: 1st value def} that $C_{\frac{1}{2},0}(n)\leq 1$ for all $n\ge 1$. Also, using binomial expansion, one can show from \eqref{eq: AB def} that $1+B_{n,\frac{1}{2}}-A_{n,\frac{1}{2},0}< \frac{1}{20n^4}$ for $n\ge 3$.  Using \eqref{eq: gen def} we get $-\Delta_{\frac{1}{2},0}(2)=(11\sqrt{2\pi}/32)-(107\sqrt{3\pi}/384)<1/80.$ Assume $-\Delta_{\frac{1}{2},0}(n)< 1/20 n^2$ for $n\ge 2$. Then using \eqref{eq: recursion final} we have
    \begin{align*}
        -\Delta_{\frac{1}{2},0}(n+1)&=(1+B_{n+1,\frac{1}{2},}-A_{n+1,\frac{1}{2},0})C_{\frac{1}{2},0}(n+1)-B_{n+1,\frac{1}{2},}\Delta_{\frac{1}{2},0}(n)\\
        &< \frac{1}{20}\left(\frac{1}{(n+1)^4} +\frac{1}{n^2}\l(1-\frac{2}{n+2}\r)^{\frac32} \right).
    \end{align*}
One can show that the above term is bounded by $\frac{1}{20(n+1)^2}$. Therefore the proof is complete by induction.
\end{proof}

\subsection*{Acknowledgement.}
  The second author acknowledges the support of Indian Institute of Science Education and Research Kolkata, through Start-up Grant and Department of Science and Technology (DST), India, through the INSPIRE Faculty Fellowship, IFA22-MA176.
 \bibliography{bibliography}

@article{abreu2026recurrence,
  title={A recurrence relation for the average singular value of a complex {G}aussian random matrix},
  author={Abreu, L. D.},
  journal={Complex Analysis and Operator Theory},
  volume={20},
  number={6},
  pages={159},
  year={2026},
  publisher={Springer}
}

@book{forrester2010log,
  title={Log-gases and random matrices (LMS-34)},
  author={Forrester, P. J.},
  year={2010},
  publisher={Princeton university press}
}

@article{H26,
    author = {{Hutník}, O.},
    title ={The Average Singular Value of a Complex {G}aussian Random Matrix Strictly Decreases with Dimension},
    note ={\url{
https://doi.org/10.48550/arXiv.2608.12147
}},
    year = {2026}
}

@article{H26N,
    author = {{Hutník}, O.},
    title ={Dimension Monotonicity in Laguerre Ensembles {I}: Fractional Moments and Shape Transitions in the Unitary Case},
    note ={\url{
https://doi.org/10.48550/arXiv.2608.12147
}},
    year = {2026}
}

@article{H26R,
    author = {{Hutník}, O.},
    title ={Dimension Monotonicity in Laguerre Ensembles {II}: Average Singular Values and the Rectangularity Transition in the Orthogonal Case},
    note ={\url{
https://arxiv.org/abs/2608.12151
}},
    year = {2026}
}

@article{AP26,
    author = {Abreu, L.D. and Patil, P.},
    title ={Non-asymptotic bounds for the average singular value of a complex Gaussian matrix},
    note ={\url{
https://arxiv.org/abs/2609.07802
}},
    year = {2026}
}

@article{BD26,
    author = {Baslingker, J. and Dan, B.},
    title ={On the monotonicity of average singular values of complex Gaussian random matrices},
    note ={\url{
https://arxiv.org/abs/2608.27532
}},
    year = {2026}
}

@article {BKS16,
    AUTHOR = {Bandeira, A. S. and Kennedy, C. and Singer, A.},
     TITLE = {Approximating the little {G}rothendieck problem over the
              orthogonal and unitary groups},
   JOURNAL = {Math. Program.},
  FJOURNAL = {Mathematical Programming},
    VOLUME = {160},
      YEAR = {2016},
    NUMBER = {1-2},
     PAGES = {433--475},
      ISSN = {0025-5610,1436-4646},
   MRCLASS = {68W25 (90C20 90C27)},
  MRNUMBER = {3555395},
MRREVIEWER = {Christiane\ Tammer},
       DOI = {10.1007/s10107-016-0993-7},
       URL = {https://doi.org/10.1007/s10107-016-0993-7},
}

@article {HT03,
    AUTHOR = {Haagerup, U. and Thorbj{\o}rnsen, S.},
     TITLE = {Random matrices with complex {G}aussian entries},
   JOURNAL = {Expo. Math.},
  FJOURNAL = {Expositiones Mathematicae},
    VOLUME = {21},
      YEAR = {2003},
    NUMBER = {4},
     PAGES = {293--337},
      ISSN = {0723-0869},
   MRCLASS = {46L54 (15A52 33C45 60E05 82B31)},
  MRNUMBER = {2022002},
MRREVIEWER = {Oleksiy\ Khorunzhiy},
       DOI = {10.1016/S0723-0869(03)80036-1},
       URL = {https://doi.org/10.1016/S0723-0869(03)80036-1},
}

@article {GMM91,
    AUTHOR = {Gupta, D. P. and Ismail, M. E. H. and Masson, D.
              R.},
     TITLE = {Associated continuous {H}ahn polynomials},
   JOURNAL = {Canad. J. Math.},
  FJOURNAL = {Canadian Journal of Mathematics. Journal Canadien de
              Math\'ematiques},
    VOLUME = {43},
      YEAR = {1991},
    NUMBER = {6},
     PAGES = {1263--1280},
      ISSN = {0008-414X,1496-4279},
   MRCLASS = {33C45},
  MRNUMBER = {1145588},
MRREVIEWER = {Mizan\ Rahman},
       DOI = {10.4153/CJM-1991-072-3},
       URL = {https://doi.org/10.4153/CJM-1991-072-3},
}

@article {CMOS19,
    AUTHOR = {Cunden, F. D. and Mezzadri, F. and O'Connell,
              N. and Simm, N.},
     TITLE = {Moments of random matrices and hypergeometric orthogonal
              polynomials},
   JOURNAL = {Comm. Math. Phys.},
  FJOURNAL = {Communications in Mathematical Physics},
    VOLUME = {369},
      YEAR = {2019},
    NUMBER = {3},
     PAGES = {1091--1145},
      ISSN = {0010-3616,1432-0916},
   MRCLASS = {60B20 (33C45)},
  MRNUMBER = {3975863},
       DOI = {10.1007/s00220-019-03323-9},
       URL = {https://doi.org/10.1007/s00220-019-03323-9},
}
 \bibliographystyle{plain}

\end{document}